\documentclass[twoside,centertags]{amsart}
\usepackage[textwidth=3cm, textsize=small, colorinlistoftodos]{todonotes}
\usepackage{comment}

\usepackage{fancyhdr}           
\fancypagestyle{plain}{
    \lhead{}
    \fancyhead[R]{\thepage}
    \fancyhead[L]{}
    
    \fancyfoot{}
}
\usepackage{comment}

\usepackage[cmtip,color,all]{xy}
\usepackage{amsmath}
\usepackage{amsfonts}
\usepackage{amssymb}
\usepackage{amsthm}
\usepackage{graphicx}
\usepackage{mathrsfs}
\usepackage{hyperref}
\usepackage{tikz-cd}
\usepackage{xcolor}
\usepackage{bbold}
\usepackage{pifont}
\usepackage{mathtools}

\newtheorem{theorem}{Theorem}[subsection]

\newtheorem{proposition}[theorem]{Proposition}
\newtheorem*{theorem*}{Theorem}

\theoremstyle{definition}
\newtheorem{definition}[theorem]{Definition}

\theoremstyle{definition}

\usepackage{geometry}
\usepackage{setspace}
\usepackage{stmaryrd,Commons,xcolor,soul, tikz}
\usetikzlibrary{cd}
\usepackage{circuitikz}
\usepackage[style = alphabetic, doi = false, url=false, isbn = false]{biblatex}
\usepackage{hyperref}
\usepackage{xurl}
\hypersetup{breaklinks=true}
\renewbibmacro{in:}{}
\DeclareFieldFormat[article]{title}{\mkbibemph{#1}}
\DeclareFieldFormat[misc]{title}{\mkbibemph{#1}}
\DeclareFieldFormat{journaltitle}{#1}
\definecolor{winered}{rgb}{0.5,0,0}
 \usepackage{hyperref}
 \hypersetup{citecolor = winered, linkcolor = winered, menucolor = winered, colorlinks = true}

\allowdisplaybreaks
\graphicspath{{./images/}}

\newcommand{\C}{\mathcal{C}}

\newcommand{\s}{\mathsf{Set}}

\newcommand{\dgC}{\mathsf{dgCoalg}}
\newcommand{\dgA}{\mathsf{dgAlg}}
\newcommand{\Og}{\mathcal{O}_G}
\newcommand{\Cobar}{\operatorname{Cobar}}
\newcommand{\Barr}{\operatorname{Bar}}

\makeatletter
\newcommand{\colim@}[2]{%
  \vtop{\m@th\ialign{##\cr
    \hfil$#1\operator@font colim$\hfil\cr
    \noalign{\nointerlineskip\kern1.5\ex@}#2\cr
    \noalign{\nointerlineskip\kern-\ex@}\cr}}%
}
\newcommand{\colim}{%
  \mathop{\mathpalette\colim@{\rightarrowfill@\scriptscriptstyle}}\nmlimits@
}
\renewcommand{\varprojlim}{%
  \mathop{\mathpalette\varlim@{\leftarrowfill@\scriptscriptstyle}}\nmlimits@
}
\renewcommand{\varinjlim}{%
  \mathop{\mathpalette\varlim@{\rightarrowfill@\scriptscriptstyle}}\nmlimits@
}
\makeatother

\theoremstyle{definition}

\title{Koszul Duality in an Equivariant Context}

\author{Sofía Martínez Alberga}
\address{Instituto Tecnológico Autónomo de México, 
Río Hondo 1, Altavista, Álvaro Obregón, 01080 Ciudad de México, CDMX, Mexico}
\email{sofia.martinez.alberga@itam.mx}

\begin{document}
\singlespacing
\maketitle
\begin{abstract}{This article takes steps to show that an equivariant version of the well known phenomena of Koszul Duality can be formulated in a manner not considered prior: via dg (co)algebra valued functors on the orbit category. 
As such techniques used in this article mix traditional and modern categorical methods in homotopy theory.}
\end{abstract}

\section{Introduction}

In the eighties, Elmendorf proved in \cite{elm} that the category of topological spaces equipped with a group action is equivalent to the category of spaced valued presheaves on the orbit category of said group, where the group is as general as a compact Lie group.
Ever since then  there has been work to generalize this result for more generic categories. 
The first source of literature known to do this is an article of Stephan's \cite{stephan} in which they provide a criteria which guarantees that any cofibrantly generated model category satisfying some \emph{cellularity conditions}, can be made into an equivariant category, for discrete and topological groups. 
Using this criteria for a discrete group $G$, Bohmann et al. \cite{gcat} proved that a model structure could be put one the category of equivariant categories and Bergner \cite{bergner} showed that the known models for $(\infty, 1)$-categories can all be extended to equivariant versions.

The main goal of this text is to show a similar conclusion to that of Stephen's hold in the context of the categories involved in the phenonmena known as \emph{differential graded (dg) Koszul duality}, namely the categories of dg coalgebras and algebras.
Today this phenomenon is wide-reaching, but it was originally a duality theory was defined to generalize the notion of duality between symmetric and exterior algebras (associated to vector spaces) to the dg algebra setting \cite{beilinson1988koszul}. 
However, there have been many formulations of the notion in a variety of contexts, some include \cite{ginz},\cite{Kape} \cite{mirk}. 

While there are many instances of Koszul Duality, the formulation of relevance to this article is the one which relates dg algebras and dg coalgebras. 
First put into print by Positselski
\cite{posit}, the formulation manifests itself within the adjunction formed by the Cobar and Bar functors.
Classically, these functors provide a method to construct an algebra from a coalgebra and vice versa, respectively. 
More generally, these functors can be extended to the dg setting, hence functors now provide a construction that takes as input a dg coalgebra and outputs a dg algebra, and vice versa, respectively. 
This extended notion is called \emph{differential graded (dg) Koszul Duality} and is manifested in the following adjunction,
$$\Cobar:  \dgC_{k}^c \leftrightarrows \dgA_{k}^a: \Barr,$$ 
where $k$ is assumed to be a field and the categories are conilpotent dg coalgebras and augmented dg algebras.  
Moreover, this adjunction establishes an equivalence of model categories. 
It is worth noting that this duality is a homological notion that historically has been used to construct well-behaved resolutions of algebras for the purpose of computing derived functors. 
Hence we expect the results of this article will provide the framework do the same in the equivariant setting, as was done in \cite{scull}. 
Many invariants for topological spaces possess the structure of a dg algebra \cite{sullivan1977} or dg coalgebra \cite{goerss1995simplicial}, \cite{RR22}, so the framework presented here sheds some light as to how to develop equivariant analogues as was achieved in \cite{MR}.
The novelty of the present article is two-fold: firstly, we obtain the desired fixed-point model structure on the category of equivariant dg algebras by other means than Stephan's framework and secondly, we formalize an equivariant version of Koszul Duality.

The organization of the present article is as follows. In the following section we establish the necessary algebraic and categorical definitions to formulate our framework.
We explain in the third section how to achieve the appropriate equivariant model structures that, despite the failing to show that the cellularity conditions hold for one of the categories, specifically the category of augmented dg algebras. 
Lastly, in Section 4 we define the equivariant Cobar and Bar functors and prove that they are the appropriate equivariant analogues. 
Furthermore, we show that the equivariant adjunction forms a Quillen equivalence. 

\subsection*{Acknowledgements:} The author would like to thank Manuel Rivera for the inspiration of this article and all the fruitful conversations pertaining to the contents. Furthermore, the author would like to thank José Manuel Moreno-Fernández for similarly fruitful conversations on topics related to the present material. 

\section{Preliminaries}

\subsection{Algebras and Coalgebras}
Let $A$ be an associative unital algebra over a field $k$.
An augmentation $\alpha$ on $A$ is a unital $k$-algebra homomorphism $\alpha : A \rightarrow k$,  so $\alpha (1) = 1$, and the existence of an augmentation implies that $1 \neq 0$ in $A$. 
Hence, $\alpha$ is a surjective map of algebras.
We denote by $A^+ = ker(\alpha) \subset A$ the augmentation ideal, meaning $A^+$ is a two-sided ideal in $A$, and $A= k \oplus A^+$ as a $k$-vector space.
Furthermore as a $k$-algebra homomorphism, $\alpha: A \rightarrow k$ endows the one-dimensional $k$-vector space $k$ with left and right $A$-module structures.

A dg algebra $A^{\bullet}=(A, d)$ over a field $k$. 
is a graded $k$-algebra, $A=\oplus_{i} A^i$ and with a differential $d: A \longrightarrow A$ which is a homogeneous $k$-linear map that is cohomologically graded of degree 1 and satisfies Leibniz rule:
    \[ d(a\cdot b) = (da)\cdot b+(-1)^{\mid a\mid }a\cdot (db),\]
and well as $d^2=0$. 
An augmentation $\alpha$ on a dg algebra $A^{\bullet}$ is a unital dg-algebra homomorphism $\alpha: A^{\bullet} \longrightarrow k$. 
This implies,
\begin{itemize}
    \item $\alpha(a)=0,$
for all $a \in A^i, i \neq 0, $
\item $\alpha(a b)=\alpha(a) \alpha(b)$ for all $a, b \in A,$ and $ \alpha(1)=1,$
\item $\alpha(d(a))=0$ for all $a \in A^{-1}$. 
\end{itemize}
We will denote by $A^{\bullet,+}=\operatorname{ker}(\alpha) \subset A^{\bullet}$ the augmentation ideal. 

Moving on to coalgebras, we mention that a standard reference sources on coalgebras are \cite{88}, \cite{93}, and \cite{21}.
A coassociative, counital coalgebra $C$ over a field $k$ is a $k$-vector space endowed with $k$-linear maps of comultiplication and counit,
$$\mu : C \rightarrow C\otimes_k C \qquad  \qquad  \epsilon: C \rightarrow k,$$
satisfying the following coassociativity and counitality axioms: 
\begin{itemize}
    \item $(\mu \otimes id_C )\circ \mu = (id_C \otimes \mu)\circ \mu$,
    \item $(\epsilon \otimes id_C )\circ \mu = id_C = (id_C \otimes \epsilon)\circ \mu.$
\end{itemize}
A coaugmentation $\gamma$ on a coalgebra $C$ is a  coalgebra homomorphism $\gamma: k \longrightarrow C_{\bullet}$. 

A dg-coalgebra $C_\bullet = (C, d)$ over $k$ is a graded coalgebra over $k$, $C= \oplus_i C_i$,  endowed with a coalgebra structure such that both the comultiplication map $\mu_\bullet$ and the counit map $\epsilon_\bullet$ are morphisms of chain complexes of vector spaces. 
Specifically they are homogeneous linear maps of degree 0 commuting with the differentials.
Here the standard induced differential, 
$$d(c' \otimes c'') = d(c') \otimes c'' + (-1)^ic' \otimes d(c''),$$  
for all $c' \in C_i$ and $c'' \in C_j$ is presumed on the tensor product of complexes $C_\bullet \otimes_k C_\bullet$, while the differential on the $k$-vector space $k$ is zero. 
Note that it will sometimes be useful to think of dg coalgebras as chain complexes of $k$-vector spaces endowed with coalgebra structure.

\subsection{Model Structures on Dg Algebras and Dg Coalgebras }
%It turns out that the adjunction induced at the beginning of the section is a Quillen equivalence between appropriate model category structures.
Now we will discuss the model structures necessary to formulate dg Koszul duality. 

Denote by $\dgA_k$ the category of dg-algebras over $k$ with the usual dg-algebra homomorphisms and let $\dgA_k^a$ the category of augmented dg-algebras with
dg-algebra homomorphisms preserving the augmentations.
Similarly, denote by $\dgC_k$ the category of dg-coalgebras over $k$ with morphisms of dg coalgebras and by $\dgC_k^{coa}$ the category of coaugmented.
Let $\dgC_k^c \subset \dgC_k^{coa}$ denote the full subcategory whose objects are the conilpotent dg-coalgebras.

The relevant model category structure on dg algebras has been described originally in \cite{Jar97}. 
Citing \cite{posit} , we provide a more digested discussion. The category $\dgA_k$ has a standard model category structure where the quasi-isomorphisms are the weak equivalences and the surjective dg-algebra maps namely, dg-algebra maps that are surjective as maps of graded vector spaces, are the fibrations. 
One can check that from the explicit descriptions of the classes of cofibrations and trivial cofibrations in $\dgA_k$ which can be found in the formulation of  Theorem 9.1(a) in \cite{pos}, that the model category $\dgA_k$ is cofibrantly generated. 
From the same reference we can conclude that the category of $\dgA_k^{a}$ possesses the same structure.  

Since it will be relevant later, we recall the generating cofibrations and acyclic cofibrations are as outlined in Theorem 3.5.2 in \cite{rr} for $\dgA_R$. 
Suppose that \(R\) is a commutative ring and let \(T_n(x)\) denote the free dg algebra generated by an element \(x\) in degree \(n\), and \(S_n(x)\) the free graded \(R\)-algebra \(R[x]\) generated by an element \(x\) in degree \(n\) and equipped with the trivial differential. 
Then the maps,
\[R\rightarrow T_n(x),\] 
for \(n\in \mathbb{Z}\),
generate the trivial cofibrations, and the maps,
\[R \rightarrow T_n(x), \quad
R \rightarrow S_n(x), \quad
S_{n-1}(dx) \rightarrow T_n(x), \]
for \(n \in \mathbb{Z}\), generate the cofibrations.

The category $\dgC_k^c$ also has a natural model structure described in \cite{hinich}.
The weak equivalences are the morphisms which get inverted when one inverts the filtered quasi-isomorphisms and the cofibrations are the injective dg-coalgebra maps namely, the dg-coalgebra maps that are injective as maps of graded vector spaces. Explicit descriptions of all the classes of morphisms involved can also be found in \cite{pos}.
Furthermore, the model category $\dgC_k^c$ is cofibrantly generated and in fact, the
injective morphisms of finite-dimensional dg-coalgebras are the generating cofibrations.
This is  due to the fact that any dg-coalgebra is the directed union
of its finite-dimensional dg-subcoalgebras. 
The injective filtered quasi-isomorphisms strictly compatible with the filtrations, acting between finite-dimensional dg-coalgebras, are the generating trivial cofibrations. 
Soon we will see that this model structure on dg coalgebras may be understood as left transferred from dg algebras along the Cobar-Bar adjunction, defined in the next section.

\subsection{A Review of DG Koszul Duality}
The notion of dg Koszul duality was first put into print by Positselski \cite{pos} and is given by the Cobar and Bar functors defining the adjunction
$$\operatorname{Cobar}:  \dgC^c_{k} \leftrightarrows \dgA^a_{k}: \operatorname{Bar},$$ 
between the categories of dg conilpotent coassociative coalgebras and of dg augmented associative algebras.
We explicitly define the $\Cobar$ functor, in the following terms: given a conilpotent dg coalgebra, \(N_{\bullet} \in \dgC^c_R\), we produce an augmented dg algebra by constructing the tensor algebra: 
    \[T(s^{-1}N_{\bullet} ) = R \oplus ^{-1}N \oplus (s^{-1}N_{\bullet}  \otimes s^{-1}N_{\bullet} ) \oplus (s^{-1}N_{\bullet}  \otimes s^{-1}N_{\bullet}  \otimes s^{-1}N_{\bullet} ) \oplus ..., \] 
where \(s^{-1}\) denotes the functor which shifts the grading by -1, namely, \((s^{-1}N_{\bullet} )_i = N_{i+1}\).
The differential is then defined as the linear extension of the map
\[-s^{-1} \circ \partial \circ s^{+1} + (s^{-1} \otimes s^{-1}) \circ \delta \circ s^{+1} : s^{-1}N_{\bullet}  \rightarrow T(s^{-1}N_{\bullet} ),\]
to all of \(T(s^{-1}N_{\bullet} )\). 
Lastly, the dg algebra \(\Cobar(N_{\bullet} )\) is equipped with a natural augmentation \(\Cobar(N_{\bullet}) \rightarrow R\) given by the canonical projection map.

The bar-construction, \(\operatorname{Bar}\left(A^{\bullet}\right)\), of an augmented dg-algebra \(\left(A^{\bullet}, \alpha\right)\) is the bicomplex,
\[
k \stackrel{0}{\longleftarrow} A^{\bullet,+} \longleftarrow A^{\bullet,+} \otimes_k A^{\bullet,+} \longleftarrow A^{\bullet,+} \otimes_k A^{\bullet,+} \otimes_k A^{\bullet,+} \longleftarrow \ldots,\]
with the bar differential $\partial$ given by extending the following formula: 
\[\partial(a \otimes b)=a b, \quad \partial(a \otimes b \otimes c)=a b \otimes c-a \otimes b c,\]
for for all $a, b, c, \in A^{\bullet,+}$, and and the differential $d$ induced by the differential $d$ on $A^{\bullet}$. 

The differential $d$ on $\operatorname{Bar}\left(A^{\bullet}\right)$ acts on every tensor power $(A^{\bullet,+})^{\otimes n}$ of the augmentation ideal $A^{\bullet,+}$ by the tensor product of $n$ copies of the differential $d$ on $A^{\bullet,+}$ in the sense of the tensor product of complexes.

It should be noted that the total complex of the bar construction $\operatorname{Bar}\left(A^{\bullet}\right)$ is produced by taking infinite direct sums along the diagonals.
The underlying graded vector space of $\operatorname{Bar}\left(A^{\bullet}\right)$ 
is endowed with the graded coalgebra structure of the tensor coalgebra $\operatorname{Bar}(A)= \emph{\rotatebox[origin=c]{180}{T}}\left(A^{+}[1]\right)$.
Here the shift of cohomological grading is 1 in the latter formula. 
This is a way to express the passage to the total complex of the bicomplex $\operatorname{Bar}\left(A^{\bullet}\right)$, with its total grading equal to the grading induced by the grading of $A$ minus the grading by the number of the tensor factors. This makes $\operatorname{Bar}\left(A^{\bullet}\right)$ a dg-coalgebra.

\section{Checking Equivariant Model Structures Exist}
For the purposes of this write-up, we will take the group $G$ to be a topological group and $\C$ to be a cofibrantly generated model category. 
Furthermore, we drop the bullet notation moving forward for dg objects to reduce heavy notation. 
\begin{definition}
    The category of $G$-objects in $\C$, denoted by $G$-$\C$, the category of functors. 
\end{definition}
For our purposes $\C$ will either be $\dgA_k^a$ or $\dgC_k^c$. 
To simplify notion, let $\Og$-$\dgA_k^a$ and $\Og$-$\dgC_k^c$ be the functor categories Fun$(\Og^{op}, \dgA_k^a)$ and Fun$(\Og^{op}, \dgC_k^c)$, respectively.
Something to note is that when inducing a $G$-equivariant model structure on a given cofibrantly generated model category, rather than just a model structure on the objects equipped with $G$-actions, it is necessary to define the weak equivalences on fixed point objects.
Fixed point objects are defined in full generality in the following manner: 
\begin{definition}
    Given any subgroup $H$ of $G$, define the $H$-fixed point functor, denoted by $(-)^H$, as the composite 
    $$G\text{-}\C \xrightarrow[]{\operatorname{restrict}} H\text{-}\C  \xrightarrow[]{\operatorname{lim}}  \C .$$
\end{definition}
Given this the \emph{fixed point model structure} on $G$-$\C$, if it exists, has as weak equivalences, respectively fibrations, the maps $X \rightarrow Y$ such that, for every subgroup $H$ of $G$, the map $X^H \rightarrow Y^H$ is a weak equivalence, respectively fibration, in $\C$.

Since $\C$ is a model category, it is complete and cocomplete so we can think of it as being tensored and cotensored over the category of sets, in the following way. 
Let $X$ and $Y$ be objects of $\C$ and $S$ a set. 
Then there is a tensor and cotensor
$$S \otimes X = \coprod_S X \quad \quad [S,Y] = \prod_S Y$$
such that
$$\operatorname{Hom}_\C( S\otimes X, Y ) \cong \operatorname{Hom}_{\s}(S, \operatorname{Hom}_\C(X,Y)) \cong \operatorname{Hom}_\C (X, [S,Y]).$$

Stephan \cite{stephan} provides a criteria for checking when a cofibrantly generated category admits the fixed point model structure. 
In particular, its when the following so-called \emph{cellularity conditions} hold.
\begin{definition}\label{cellular}
    The $H$-fixed point functors $(-)^H$ are said to be cellular, or satisfy the cellularity conditions, if for any subgroup $H$ of $G$ the follow holds: 
\begin{enumerate}
    \item the functor $(-)^H$ preserves pushouts of diagrams where one arrow is of the form 
    $$id_{G/L} \otimes f : G/L \otimes X \rightarrow G/L \otimes Y$$
    for some subgroup $L$ of $G$ and $f$ a generating cofibration of $\C$
    \item the functor $(-)^H$ preserves filtered colimits of diagrams in $G$-$\C$
    \item for any subgroup $L$ of $G$ and object $X$ in $\C$, the induced map 
    $$(G/H)^L \otimes X \rightarrow (G/H \otimes X)^L$$
    is an isomorphism in $\C$
\end{enumerate}
\end{definition}
We will make regular use of following theorems so we recall them here.
\begin{theorem}\label{t2.1}(Thereom 3.17 in \cite{stephan},  Theorem 1.8, Chapter 3 in \cite{MM02})
    If $\C$ is a cofibrantly generated model category, with a generating set of cofibrations denoted $I$ and a generating set of acyclic cofibrations dentored $J$, such that $G$-$\C$ has cellular fixed point functors, then $G$-$\C$ is a cofibrantly generated model category with generating cofibrations
    $$I_G = \{G/H \otimes i \}_{i \in I, H \leq G}, $$
    generating acyclic cofibrations
    $$J_G = \{G/H \otimes j\}_{j \in J, H \leq G},$$ 
    and weak equivalences are maps of $G$-objects which are weak equivalences in $\C$ after passage to $H$ fixed points for all $H \leq G$.
\end{theorem}
Now we establish the model category structure of the category of presheaves $\Og^{op}$-$\C$, where $\Og$ denotes the orbit category of the group $G.$
The category $\Og$ is the full subcategory of the category of $G$-sets with objects being orbits, $G/H$, for $H$ a subgroup of $G$. 
Therefore, $\Og^{op}$-$\C$ is the category of contravariant functors from $\Og$ to $\C$ with morphisms being natural transformations.

Using Theorem 11.6.1 in \cite{hird}, any category of diagrams in the cofibrantly generated model category $\C$ admits the \emph{projective model structure} and is again cofibrantly generated.
Suppose that $\C$ has generating cofibration set $I$ and generating acyclic cofibration set $J$.
Then the category of orbit diagrams
$\Og^{op}$-$\C$ has generating cofibrations and generating cofibrations
$$I_{\Og^{op}}:=\{ \Og^{op}(G / H, -) \otimes i\}_{i \in I, H \leq G},$$
and generating acyclic cofibrations
$$J_{\Og^{op}}:=\{\Og^{op}(G / H, -) \otimes j\}_{j \in J, H \leq G} .$$
It is worth noting that for a generating cofibration \(i: X \rightarrow Y\) in \(\C\), one can think of \(\Og^{op}(G / H, -) \otimes i\) as a natural transformation, from a functor defined by taking \(G/L\) to \((G/L)^H \otimes X\)  to another functor from to \(G/L\) to \((G/L)^H \otimes Y\). 
This natural transformation on components is a map from \((G/L)^H \otimes X\) to \((G/L)^H \otimes Y\). 
In the projective model structure, the weak equivalences and fibrations are defined as the object-wise weak equivalences and fibrations.

\subsection{The categories $G$-$\dgC_k^c$ and $\Og$-$\dgC_k^c$} 
Note that in this subsection since we are working with dg coalgebras we will take the tensor product to be the usual one.
\begin{proposition}\label{p4.3} The category $G$-$\dgC_k^c$ admits the fixed point model structure. 

\end{proposition}
\begin{proof}
To prove the proposition we need to show the analogous cellularity conditions recalled in Definition \ref{cellular} hold. 
We remark that the proof of this Proposition is analogous to that of Proposition  3.2.1 in \cite{MR}.
\end{proof}
Next we establish the model structure on the presheaf category. 
\begin{proposition}\label{p3.1.2} The category $\Og$-$\dgC_k^c$ has a model structure where a map $\beta :X \rightarrow Y$ in $\Og$-$\dgC_k^c$  is a weak equivalence if $\beta(G/H)$ is a weak equivalence in $\dgC_k^C$ and is a fibration if $\beta(G/H)$ is a fibration in $\dgC_k^c$ for all $H < G.$
\end{proposition}
 The above statement follows from Proposition A.2.8.2 in \cite{Lur09}.
 With the previous two propositions we can deduce the following ``Elemendorf style'' theorem for conilpotent dg coalgebras. 
\begin{theorem}\label{thm:3.1.3}
    There is a Quillen equivalence of model categories 
    $$ \Theta: \Og \text{-} \dgC_k^c \rightleftarrows  G\text{-} \dgC_k^c : \Phi$$
    where $\Phi(X) = X(G/e)$ and $\Theta(Y) = \underline{Y}$ with $\underline{Y}$ being the functor that takes $G/H$ to the connected simplicial cocommutative coalgebra $Y^H$.
\end{theorem}

This proof is also analogue to that of simplicial coalgebra setting namely Theorem 3.1.3 in \cite{MR}.

\subsection{The categories $G$-$\dgA_k^a$ and  $\Og$-$\dgA_k^a$}
In this subsection, we investigate how to obtain the desired fixed points model structure on $G$-$\dgA_k^a$. 

Note that in this subsection, since we are working with augmented dg algebras, we will take the tensor product to be the usual tensor product.
\begin{proposition}\label{p4.1} The category $\dgA^a_k$ admit condition 3 and 1 in Definition \ref{cellular}.

\end{proposition}
\begin{proof}   
We begin by showing the third cellularity condition.
Given a subgroup \(L \leq G\) and \(A \in \dgA_k^a\) we want to show that the induced map 
\[\left(\coprod_{({G/L})^H} A \right) \rightarrow  \left(\coprod_{G/L} A\right)^H, \]
is an isomorphism.
Note that the coproduct in \(\dgA_R^a\) is the usual tensor product of algebras. 
So the action of \(H\) on \(\otimes_{G/L} A\) is given by permuting the copies of \(A\), since each \(A\) itself has trivial \(G\)-action.
Therefore, this action is determined precisely by the action of \(H\) on \(G/L\). 
The desired isomorphism follows.

To prove 1. we use condition 3.
Consider the pushout in \(G\)-\(\dgA_k^a\)
\[\begin{tikzcd}
 \otimes_{G/L} R  \arrow[d, "\otimes_{G/L} i"'] \arrow[rr, "f"] &  & A \arrow[d] \\
\otimes_{G/L} T_n(x)\arrow[rr]  &  & D,
\end{tikzcd}\]
where \(D = A \otimes \otimes_{G/L} T_n(x)\) 
and \(i\) is a generating cofibration as described earlier. 
Note that \(D\) is given by concatenating \(G/L\)-many copies tensor copies of \(T_n(x)\) to \(A\). 
This can be thought of a equivariantly gluing ``\(n\)-disks'' to the \(n\)-th level of \(A\) and their boundaries attached at level \(n-1\) of \(A\).
Now we apply the fixed point functor \((-)^H\) to the upper horizontal arrow and left vertical arrow and then take the pushout, so we get the diagram below 
\[\begin{tikzcd}
 (\otimes_{G/L} R)^H  \arrow[d, "(\otimes_{G/L} i)^H"'] \arrow[rr, "f^H"] &  & A^H \arrow[d] \\
(\otimes_{G/L} T_n(x))^H\arrow[rr]  &  & D^H,
\end{tikzcd}\]
but using condition 3. the diagram becomes 
\[\begin{tikzcd}
 \otimes_{(G/L)^H} R  \arrow[d, "\otimes_{(G/L)^H} i"'] \arrow[rr, "f^H"] &  & A^H \arrow[d] \\
\otimes_{(G/L)^H} T_n(x)\arrow[rr]  &  & A^H \otimes (\otimes_{(G/L)^H} T_n(x)). 
\end{tikzcd}\]
So the resulting pushout is given by concatenating to \(A_n^H\) the appropriate copy of \(T_n(x)\) indexed by \((G/L)^H\). 

If we inspect the \(H\)-fixed points of \(D\) we find that they coincide with pushout above because \(D^H\) is only concerned with the copies of \(T_n(x)\) indexed by \((G/H)^L\) and glues them equivariantly to the \(H\)-fixed points \(A\) at level \(n.\)
\end{proof}

We are unable to show explicitly that condition 2 in the Definition \ref{cellular} holds for $\dgA_k^a$, but we do not mean to imply that Theorem \ref{cellular} does not hold for $\dgA_k^a$. 
Given this reality we recall the subsequent theorem so that at least we may obtain a right induced model structure on fibrant objects. 
After showing the induced model structure exists, we  argue why it coincides with the fixed point model structure. 
\begin{theorem}(Theorem A in \cite{Maru}) \label{thm:maru}
Let $\left(\mathcal{M}, \operatorname{Cof}_{\mathcal{M}}, \operatorname{Fib}_{\mathcal{M}}, \mathcal{W}_{\mathcal{M}}\right)$ be a combinatorial model category, and let $\mathcal{C}$ be a locally presentable category. Suppose that we have an adjunction
$$\begin{tikzcd} 
{\C}  \arrow[rr, "F"', shift left=-2] & & {\mathcal{M}},  \arrow[ll, "F'"',shift right=2 ]
\end{tikzcd}$$
and that the following properties are satisfied:
\begin{enumerate}
    \item  every morphism $f \in F^{-1}\left(\mathrm{Fib}_{\mathcal{M}} \cap \mathcal{W}_{\mathcal{M}}\right)$ has a fibrant replacement $f' \in F^{-1}\left(\mathcal{W}_{\mathcal{M}}\right)$,
    \item for every object $X$ such that $FX$ is fibrant, there is a factorization of the diagonal morphism $X \xrightarrow{w} P X \xrightarrow{p} X \times X$ with $w \in F^{-1}\left(\mathcal{W}_{\mathcal{M}}\right)$ and $p \in F^{-1}\left(\mathrm{Fib}_{\mathcal{M}}\right)$.
\end{enumerate}
Then, there exists a combinatorial model structure on $\mathcal{C}$ in which an object $X$ is fibrant if and only if $R X$ is fibrant in $\mathcal{M}$, and a morphism $f$ between fibrant objects is a weak equivalence (resp. fibration) if and only if $R f$ is so in $\mathcal{M}$.
\end{theorem}

Translating this theorem to our context, we have the following adjunction: 
\begin{equation} \label{adj}
    \begin{tikzcd} 
{G\text{-}\dgA_k^a}  \arrow[rr, "(-)^H"', shift left=-2] & & {\dgA_k^a}.  \arrow[ll, "G/H \otimes -"',shift right=2 ]
\end{tikzcd}
\end{equation}
It is worth noting that the model structure on $\dgA_k^A$ is also combinatorial \cite{RR22}. 
We will continue by showing that the conditions of this theorem are satisfied in this context in the following proof.

\begin{proof}
   
Condition 1: 
Suppose we have $f: X \rightarrow Y$ in $G$-$\dgA_k^a$ satisfying the condition above, then we want to show that 
$$\begin{tikzcd}
    X  \arrow[rr, "f"]  \arrow[d]& & Y\arrow[d]\\
    X'  \arrow[rr, "f'"] && Y',
\end{tikzcd}$$
there exists $f': X' \rightarrow Y'$ such that $(X')^H$ and $(Y')^H$ are fibrant and $(f')^H$ is a quasi-isomorphism. 
Consider the diagram below above when the $(-)^H$ is applied to $f$:
$$\begin{tikzcd}
    X^H  \arrow[rr, "f^H"]  \arrow[d]& & Y^H\arrow[d]\\
    \tilde{X}  \arrow[rr, dashed] && \tilde{Y},
\end{tikzcd}$$
where the vertical arrows are fibrant replacements in $\dgA_k$. Note that $f^H$ is assumed to be a trivial fibration.
Our goal is to construct the dashed arrow in such a way that it is a quasi-isomorphism in $\dgA_k^a.$
Recall that all objects in $\dgA_k^a$ are fibrant in the model structure prescribed so the diagram prior collapses to just the top arrow. 

A potential candidate for $f' $ is $G/H \otimes f^H$ which given the previous discussion translates to $\sqcup f^H$. Unraveling the symbols we find, 
$$\sqcup f^H: \sqcup X^H \rightarrow \sqcup Y^H,$$
so now the goal is to show that $(\sqcup f^H)^H$ is a quasi-isomorphism of dg algebras.
We know that $f^H$ is a quasi-isomorphism since it is a trivial fibration so $(\sqcup f^H)^H$ will also be a trivial fibration, since it is coproduct, indexed by cosets of $H$, of trivial fibrations.
Therefore $f'$ in the inverse image of $\mathcal{W}_{\dgA_k}$ under $(-)^H$, and hence the fibrant replacement of $f$.

Condition 2: Given any $X \in G$-$\dgA_k^a$, consider the diagram below:
$$\begin{tikzcd} 
{X}  \arrow[dr, "w"', shift left=-2] \arrow[rr, "\text{diagonal}"] & & {X \times X}  \\
& P(X) \arrow[ur, "p"', shift left=-2] 
\end{tikzcd}$$
where,
\begin{enumerate} 
\item our candidate for the path object $P(X)$ is $\operatorname{Hom}_R(C,X)$ where $C$ is the dg coalgebra as follows
$$0 \leftarrow C_0 = R[b,b_e] \leftarrow C_1 = R[e] \leftarrow C_2 =0 \cdots,$$
with the only nonzero differential being $e \mapsto b_e$. 
Notice that by construction this dg coalgebra $C$ is contractible and using the point $b$ we can define an augmentation map from $P(X)$ to the ground ring $R$ as follows: 
$$a: P(X)\rightarrow R,$$
$$p_b \mapsto 1_R,$$
where $p_b$ is the constant path at the point $b$.
The differential for the dg algebra $P(X)$ will be denoted $d$ and is defined as 
$$d \varphi = \varphi \circ \partial_C + \partial_X \circ \varphi,$$
with $\partial_C$ and $\partial_X$ being the differentials of $C$ and $X$, respectively. 
Additionally the multiplication of this dg algebra is given by the convolution, so given $\varphi$ and $\psi$ in $P(X)$ then 
$$\varphi \star \psi := \mu_X (\varphi \otimes \psi) \Delta_C,$$
where $\mu_X$ is the product in $X$ and $\Delta_C$ is the product in $C$.
\item our candidate for the map $w$ is as follows: 
$$w: X \rightarrow P(X),$$
$$x \mapsto w_x: C \rightarrow X,$$
and 
$$e \mapsto 0 ,\quad b_e \mapsto 0 \quad b \mapsto x,$$
\item our candidate for $p$ is an evaluation map as follows
$$p : P(X) \rightarrow X\times X,$$
$$ \varphi \mapsto (\varphi(b_e), \varphi(b)).$$
\end{enumerate}
If $X$ is taken to be a $G$-$\dgA_k^a$, then $P(X)$ as defined also inherits a $G$-action defined as 
$$g \cdot \varphi = g \varphi(c),$$
for all $\varphi \in P(X)$, $c\in C$, and $g \in G$.
Additionally $(X \times X, \partial_{X \times X})$, is also a $G$-$\dgA_k^a$ with the action performed component-wise

We start by checking we have the necessary structure and properties for the map $p$.
\begin{enumerate}
    \item $p$ is a $G$-map: 
    Given $\varphi \in P(X)$ and $g \in G$, we have
    $$g\cdot p(\varphi) = g (\varphi(b), \varphi(b_e)) = (g(\varphi(b), g \varphi(b_e)) = p(g \varphi).$$
    \item $p$ is a map of dg algebras: 
    Given $\varphi, \psi \in P(X)$, we have
    \begin{align*}p(\varphi \star \psi) = & p(\mu_X(\varphi \otimes \psi)\Delta_C) \\
    =&  (\mu_X(\varphi \otimes \psi)\Delta_C (b),\mu_X(\varphi \otimes \psi)\Delta_C (b_e)) \\
    = &(\varphi(b) \psi(b) , \varphi(b_e) \psi(b_e)),
    \end{align*}
    while 
       \begin{align*} \mu_{X \times X} (p(\varphi)  p (\psi) ) = & \mu_{X \times X} ((\varphi(b),\varphi(b_e), (\psi(b), \psi(b_e)) 
       =  (\varphi(b) \psi(b), \varphi(b_e) \psi(b_e)).
    \end{align*}
    \item $p$ is a chain map: Given $\varphi \in P(X)$, we have 
    \begin{align*}
        p \circ d \varphi = & (d \varphi (b), d \varphi (b_e)) \\
        = & ( \varphi \partial_C (b) + \partial_X \varphi(b), \varphi \partial_C (b_e) + \partial_X \varphi(b_e))\\
        = & (\partial_x \varphi (b) \partial_X \varphi(b_e))\\
        =& \partial_{X\times X} (\varphi(b), \varphi(b_e))\\
        =& \partial_{X\times X} \cdot p \varphi.
    \end{align*}
    \item $p$ is a surjective map on fixed points: 
    Given $(x,y) \in X\times X$ we can define a map $\theta \in P(X)$ as a map
    $\theta: C \rightarrow X$
    where $b \mapsto a$, $b_e \mapsto b $, and $ e \mapsto 0.$
    Clearly, $\theta$ is a linear map, hence in $P(X)$, and $p(\theta) = (x,y)$, so $p$ is surjective.
    Now suppose $H$ is a subgroup of $G$ and consider $(x,y) \in X^H \times X^H$. 
    Then, for any $h \in H$, we have $hx= x$ and $hy= y$.
    The goal is to show that the associated $\theta$ is in $P(X)^H$:
    \begin{align*}
        h(\theta(b)) & = h(x)= x = \theta(b),\\
        h(\theta(b_e)) & = h(y)= y = \theta(b_e),\\
        h(\theta(e)) & = h(0)= 0 = \theta(e).\\
    \end{align*}
    Hence $h \theta = \theta$.
\end{enumerate}
Now we check we have the necessary structure and properties for the map $w$.
\begin{enumerate}
    \item $w$ is a $G$-map:
    Given $x \in X$ and $g \in G$, then we have
    $$g \cdot w (x)(e) = g\cdot 0  = 0 \quad \text{and} \quad  w(gx)(e) = 0 , $$
   and similarly for $b_e$.
      $$g \cdot w (x)(b) = g\cdot x  = gx \quad \text{and} \quad  w(gx)(b) = gx $$
    \item $w$ is a map of dg algebras:
     Given $x, y \in X$, then we have the map
     $$w \circ \mu_X(x,y) = w(xy)= w_{xy}.$$
     So for $e \in C$ we have 
     $$w \circ \mu_X(x,y)(e) =  w_{xy}(e) = 0,$$
     on the other hand 
     \begin{align*}
    (w(x) \star w(y)) (e) = &  (w_x \star w_y)(e) \\
    =& \mu_X(w_x \otimes w_y)\Delta_C(e)\\
    =& \mu_X(w_x \otimes w_y)(e \otimes e)\\
    =& \mu_X(w_x \otimes w_y)(e \otimes e)\\
    =& \mu_X(w_x(e) \otimes w_y(e))\\
    =& \mu_X(0 \otimes 0) = 0 .
     \end{align*}
    A similar argument holds for $b_e$. 
    For $b \in C$, we have 
     $$w \circ \mu_X(x,y)(b) =  w_{xy}(b) = xy,$$
     on the other hand,
     \begin{align*}
    (w(x) \star w(y)) (b) = &  (w_x \star w_y)(b) \\
    =& \mu_X(w_x \otimes w_y)\Delta_C(b)\\
    =& \mu_X(w_x \otimes w_y)(b \otimes b)\\
    =& \mu_X(w_x \otimes w_y)(b \otimes b)\\
    =& \mu_X(w_x(b) \otimes w_y(b))\\
    =& \mu_X(x \otimes y) = xy.
     \end{align*}
    Hence $w$ is a map of dg algebras. 
    \item $w$ is a chain map: 
     Given $x \in X$, we have 
     \begin{align*}
     d \circ w(x) (b) = &  w_x \partial_C(b) + \partial_X w_x(b)\\
       =&  w_x (0) + \partial_X (x) = \partial_X (x),
     \end{align*}
     and on the other hand we have 
      \begin{align*}
     w_x \circ \partial_x(x) (b) = & w_{\partial_X (x)}(b)
    = \partial_X (x).
     \end{align*}
    Similarly for $b_e$ we get 
     \begin{align*}
     d \circ w(x) (b_e) = &  w_x \partial_C(b_e) + \partial_X w_x(b_e)\\
       =&  w_x (0) + \partial_X (0) = 0, 
     \end{align*}
     and on the other hand we have 
      \begin{align*}
     w_x \circ \partial_x(x) (b_e) = & w_{\partial_X (x)}(b_e) =  0.
     \end{align*}
     For $e$ we have
       \begin{align*}
     d \circ w(x) (e) = &  w_x \partial_C(e) + \partial_X w_x(e)\\
       =&  w_x (b_e) + \partial_X (0)\\
       =&  w_x(b_e) + 0  = 0 +0 = 0,
     \end{align*}
     and on the other hand we have 
      \begin{align*}
     w_x \circ \partial_x(x) (e) = & w_{\partial_X (x)}(e) =  0.
     \end{align*}
    \item $w$ is a quasi-isomorphism:
    Recall that $C$ is a finite dimensional dg coalgebra which is also contractible, hence the linear dual of $C$, denoted $C^*$, is a finite dimensional dg algebra. 
    Additionally notice that $$P(X) = \operatorname{Hom}_R(C, X) = C^* \otimes X,$$
    and the last dg algebra in this list of equalities is quasi-isomorphic to $X$ via the map $w$.
    %Consider the induced map of $w$ on homology  $$H_*(w): H_*(X) \rightarrow H_*(P(X)),$$ we will show that this map is an isomorphism.  To show that the map is injective, assume that for $\bar{x}$ and $\bar{y}$ in $H_*(X)$ $$H_*(w)(\bar{x}) = H_*(w)(\bar{y}),$$ we will show that $\bar{x} = \bar{y}$.\begin{align*} \bar{x}=& w_{\bar{x}}(b_0)\\=& H_*(w)(\bar{x})(b_0)\\ =& H_*(w)(\bar{y})(b_0)\\=& w_{\bar{y}}(b_0)\\ =& \bar{y} \end{align*}

\end{enumerate}
 
\end{proof}
We have shown that the conditions of Theorem \ref{thm:maru} are satisfied, so we conclude that the adjunction (\ref{adj}) holds true. 
Since $k$ is a field of characteristic zero, then all objects in $\dgA_k^a$ are fibrant \cite{posit}. 
Hence the adjunction (\ref{adj}) can be thought of as an honest right-transfer, as opposed to a fibrant-transfer.
Consequently, $G$-$\dgA_k^a$ admits the fixed point model structure, as desired, sinc the weak equivalences and fibrations given in \ref{thm:maru} in the specified context coincide. 
Note that all arguments hold in the previous proof when $k$ is a ring, yet we would not be able to deduce the previous claim.

Subsequently, we turn our attention to the category of presheaves. 
\begin{proposition}\label{p4.2}
The category $\Og$-$\dgA_R^a$ admits the projective model structure, where a map $\beta :A \rightarrow B$ in  $\Og$-$\dgA_R^a$ is a weak equivalence if $\beta(G/H)$ is a weak equivalence in $\dgA_R^a$and is a fibration if $\beta(G/H)$  is a fibration in $\dgA_R^a$ for all $H < G.$
\end{proposition}

Note that this proposition follows from Proposition A.2.8.2 in \cite{Lur09}, since the model category structure on $\dgA_k^a$ considered here satisfies the hypothesis.

Since we did not show that the cellularity conditions were satisfied, we do not conclude this subsection with the analoguous result to Theorem \ref{thm:3.1.3}, yet we were still be able to achieve the fixed point model structure we desired and will still be able to deduce the appropriate equivariant analogue for Koszul duality.

\section{Equivariant Bar and Cobar}

We want to show an adjunction like the one below exists 
$$G\text{-}\Cobar: G\text{-}\dgC_{k}^c \leftrightarrows G\text{-}\dgA_{k}^a: G\text{-}\Barr,$$
when $k$ is a field of characteristic 0, given the model structures considered in this article. 
The goal is to define the functors $G\text{-}\Cobar$ and $ G\text{-}\Barr$ by means of the diagram below and then show that this adjunction forms a Quillen equivalence.
Consider the diagram: 
$$\begin{tikzcd}
\dgC_k^c \arrow[rrr, "\Cobar", shift right=-2] \arrow[dd, "G/H\otimes -"', shift right] &  &  &
\dgA_k^a  \arrow[dd, "G/H\otimes -"', shift right] \arrow[lll, "\Barr", shift right=-2] \\
& & &   \\
G\text{-}\dgC_k^c 
\arrow[uu, "(-)^H"', shift right] \arrow[rrr, "G\text{-}\Cobar", shift right=-2]        &  &  & 
G\text{-}\dgA_k^a.  \arrow[uu, "(-)^H"', shift right]   \arrow[lll, "G\text{-}\Barr", shift right=-2] 
\end{tikzcd}$$
Now we establish that $G\text{-}\Barr_H: = G/H \times - \circ \Barr \circ (-)^H$ and $G\text{-}\Cobar_H: = G/H \times - \circ \Cobar \circ (-)^H$.
It is worth noting that the Cobar functor is a left Quillen functor  so it preserves colimits and that the Bar functor is a right Quillen functor so it preserves limits. 
Recall all the functors are defined for any subgroup $H$ of $G$, so moving forward we will take $G\text{-}\Barr := G\text{-}\Barr_e$ and $G\text{-}\Cobar:= G\text{-}\Cobar_e$ to simplify notation. 

\begin{theorem}
     The functors $(G\text{-}\Cobar, G\text{-}\Barr)$ form a Quillen equivalence.
\end{theorem}
\begin{proof}
    We begin by showing that $(G\text{-}\Cobar, G\text{-}\Barr)$ form a Quillen adjunction.
    To do this we will show that $G\text{-}\Barr$ preserves fibrations and acyclic fibrations so consider $\varphi$ and $\tilde{\varphi}$ respectively in $G\text{-}\dgA_k^a$.
    Namely, we will show that $G\text{-}\Barr(\varphi)$ is a fibration and $G\text{-}\Barr(\tilde{\varphi})$ is an acyclic fibration in $G\text{-}\dgC.$
    By definition of the fixed point model structure on $G\text{-} \dgA_k$, $\varphi$ is (acyclic)  fibration if for every subgroup $H$ of $G$ the induced morphism $\varphi^H$ is a (acyclic) fibration in $\dgA_k$.
    The functor $\Barr$ is a right Quillen functor, so it preserves (acyclic) fibrations hence $\Barr(\varphi^H)$ is a fibration in $\dgC^c_k$. 
    Lastly, we need to show that $G/H \otimes \Barr(\varphi^H)$ is a (acyclic) fibration in $G\text{-}\dgC_k^c$, and for this to be true we need $(G/H \otimes \Barr(\varphi^H))^H$ to be a (acyclic) fibration in $\dgC_k^c$.
   This is clear since $G/H\otimes-$ induces the trivial action and $\Barr(\varphi^H)$ is a (acyclic) fibration. 
   With the appropriate modifications indicated a similar argument holds for $\tilde{\varphi}$, so we conclude that $(G\text{-}\Cobar, G\text{-}\Barr)$ form a Quillen adjunction.

   Next, we show that the derived unit and counit of $(G\text{-}\Cobar, G\text{-}\Barr)$ are weak equivalences in the corresponding categories. 
   Let $D \in G\text{-}\dgC_k^c$ be a cofibrant object and consider the unit 
   $$\eta:  G\text{-}\dgC_k^c \xrightarrow[]{G\text{-}\Cobar}   G\text{-}\dgA_k^a \xrightarrow[]{G\text{-}\Barr}  G\text{-}\dgC_k^c ,$$
   where on the component of $D$ we have 
   $$\eta_D: D \mapsto (G\text{-}\Cobar(D))^f\mapsto G\text{-}\Barr( (G\text{-}\Cobar(D))^f) ,$$
   and where $(-)^f$ denotes the fibrant replacement. 
   We want to show that $\eta_D$ is a $\Cobar$-quasi-isomorphism on fixed points and to better understand the functors at play we unravel the symbols in the previous line;
      $$\eta_D: D \mapsto (G/e \otimes \Cobar(D^e))^f\mapsto G/e \otimes \Barr( ( (G/e \otimes \Cobar(D^e))^f)^e) .$$
  Recall that all objects are fibrant in $G\text{-}\dgA_k^a$. 
  % and we are also working in derived categories, so we drop the decoration on the target object of $c$ under $G\text{-}\Cobar$.  
  Having reduced to considering $G/e \otimes \Barr( ( (G/e \otimes \Cobar(D^e)))^e)$, we wish to show when it is $\Cobar$-quasi-isomorphic to $D.$
  Inspecting further we that we are taking fixed points with respect to the trivial action, namely the one induced by $G/e \otimes -,$ hence we can reduce the target object of $\eta_D$ to
   $G/e \otimes \Barr(\Cobar(D^e))$. 
   At this juncture we can utilize the fact that $(\Cobar, \Barr)$ form a Quillen equivalence so the unit is well behaved, so we get that the target object becomes $G/e \otimes (D^e)$. 
   It becomes clearer now that $D$ is $\Cobar$-quasi-isomorphic to $G/e \otimes (D^e)$, since  $G/e \otimes -$ induces the trivial action at which point we are comparing the homology of the fixed points of $D$, the homology of the domain of $\eta_D$, with the homology of the fixed points of $D$, the homology of the codomain of $\eta_D$. 

   To show that the counit is a weak equivalence (quasi-isomorphism) in $G\text{-}\dgA_k^a$ we can apply the same reasoning as with the counit.  
\end{proof}
Lastly we consider the diagram below, 
$$\begin{tikzcd}
G\text{-}\dgC_k^c \arrow[rrr, "G\text{-}\Cobar"', shift right=2] \arrow[dd, "\Phi"', shift right] &  &  &
G\text{-}\dgA_k^a  \arrow[dd, "\Phi'"', shift right] \arrow[lll, "G\text{-}\Barr"', shift right=2] \\
& & &   \\
\Og\text{-}\dgC_k^c 
\arrow[uu, "\Theta"', shift right] \arrow[rrr, "\Og\text{-}\Cobar"', shift right=2]                  &  &  & 
\Og\text{-}\dgA_k^a  \arrow[uu, "\Theta'"', shift right]   \arrow[lll, "\Og\text{-}\Barr"', shift right=2] 
\end{tikzcd}$$
By Theorem \ref{thm:3.1.3} we have that the left pair of arrows forms a Quillen equivalence.

The functors 
$\Og\text{-}\Cobar$ and $\Og\text{-}\Barr$ are defined as follows using post-composition: 
%https://mathoverflow.net/questions/136545/quillen-equivalence-of-diagram-categories this shows you can promote quillen equivalences to diagram categories
let $\underline{C}$ be in $\Og\text{-}\dgC_k^c$
$$\Og\text{-}\Cobar: \Og\text{-}\dgC_k^c \rightarrow \Og \text{-}\dgA_k^a, \qquad \qquad \underline{C} \mapsto \Og\text{-}\Cobar(\underline{C}), $$
where the new object in $\Og$-$\dgA_k^a$ is the functor determined by $\underline{C}$ and $\Cobar$
$$\Og\text{-}\Cobar(\underline{C}) : \Og^{op} \xrightarrow{\underline{C}} \dgA_k^a\xrightarrow{\Cobar} \dgA_k^a,$$
Similarly given $\underline{A} \in \Og$-$\dgA_k^a$ we can define 
$$\Og\text{-}\Barr: \Og\text{-}\dgA_k^a \rightarrow \Og \text{-}\dgC_k^c, \qquad \qquad \underline{A} \mapsto \Og\text{-}\Barr(\underline{A}), $$
where the new object in $\Og$-$\dgC_k^c$ is the functor determined by $\underline{A}$ and $\Barr$
$$\Og\text{-}\Barr(\underline{A}) : \Og^{op} \xrightarrow{\underline{A}} \dgA_k^a\xrightarrow{\Barr} \dgC_k^c.$$

Since the model structures in each presheave category is defined object-wise, it is also not hard to see that $(\Og\text{-}\Cobar, \Og\text{-}\Barr)$ form a  Quillen equivalence as well.

For the vertical pair of functors on right we define them as follows: 
$$ \Phi':  G\text{-}\dgA_k^a \rightarrow \Og\text{-}\dgA_k^a \qquad \qquad A \mapsto \Phi(A) = A(G/e)$$
and 
$$\Theta': \Og\text{-}\dgA_k^a \rightarrow G\text{-}\dgA_k^a \qquad \qquad B \mapsto \Theta(B) = \underline{B}$$
 with $\underline{B}$ being the functor that takes $G/H$ to the conilpotent dg coalgebra $B^H$.
Hence we can conclude the following result: 
\begin{theorem}
     The functors $(\Theta', \Phi')$ form a Quillen equivalence.
\end{theorem}
\begin{proof}
We deduce the pair $(\Theta', \Phi')$ form  a Quillen adjunction as they are defined in terms of the appropriate Quillen functors. 
Furthermore given that three pairs of Quillen functors each for a Quillen equivalence, we deduce the the fourth pair form a Quillen equivalence. 
\end{proof}

\newpage
\printbibliography

@article{rr,
  title={The simplicial coalgebra of chains under three different notions of weak equivalence},
  author={Raptis, George and Rivera, Manuel},
  journal={arXiv preprint arXiv:2206.06335},
  year={2022}
}

@article{stephan,
  title={On equivariant homotopy theory for model categories},
  author={Stephan, Marc},
  journal={arXiv preprint arXiv:1308.0856},
  year={2013}
}

@article{gcat,
  title={A model structure on GCat},
  author={Bohmann, Anna Marie and Mazur, Kristen and Osorno, Ang{\'e}lica M and Ozornova, Viktoriya and Ponto, Kate and Yarnall, Carolyn},
  journal={Women in Topology: Collaborations in Homotopy Theory},
  volume={641},
  pages={123--134},
  year={2015}
}

@article{elm,
  title={Systems of fixed point sets},
  author={Elmendorf, Anthony D},
  journal={Transactions of the American Mathematical Society},
  volume={277},
  number={1},
  pages={275--284},
  year={1983}
}

@book{hird,
  title={Model categories and their localizations},
  author={Hirschhorn, Philip S},
  number={99},
  year={2003},
  publisher={American Mathematical Soc.}
}

@article{sullivan1977,
  title={Infinitesimal computations in topology},
  author={Sullivan, Dennis},
  journal={Publications Math{\'e}matiques de l'IH{\'E}S},
  volume={47},
  pages={269--331},
  year={1977}
}

@article{ginz,
  title={Koszul duality for operads},
  author={Ginzburg, Victor and Kapranov, Mikhail},
  journal={Duke Math. J.},
  pages={203–272},
  year={1994}
}

@book{Lur09,
  title={Higher topos theory},
  author={Lurie, Jacob},
  year={2009},
  publisher={Princeton University Press}
}

@inproceedings{Kape,
  title={On DG-modules over the de Rham complex and the vanishing cycles functor},
  author={Bloch, Spencer and Dolgachev, Igor V and Fulton, William and Kapranov, MM},
  booktitle={Algebraic Geometry: Proceedings of the US-USSR Symposium held in Chicago, June 20--July 14, 1989},
  pages={57--86},
  year={1991},
  organization={Springer}
}

@article{posit,
  title={Differential graded Koszul duality: An introductory survey},
  author={Positselski, Leonid},
  journal={Bulletin of the London Mathematical Society},
  volume={55},
  number={4},
  pages={1551--1640},
  year={2023},
  publisher={Wiley Online Library}
}

@article{mirk,
  title={Linear koszul duality},
  author={Mirkovi{\'c}, Ivan and Riche, Simon},
  journal={Compositio Mathematica},
  volume={146},
  number={1},
  pages={233--258},
  year={2010},
  publisher={London Mathematical Society}
}

@book{pos,
  title={Two kinds of derived categories, Koszul duality, and comodule-contramodule correspondence},
  author={Positselski, L.},
  volume={212},
  number={996},
  year={2011},
  publisher={American Mathematical Society}
}

@article{Maru,
  title={Fibrantly-induced model structures},
  author={Guetta, L{\'e}onard and Moser, Lyne and Sarazola, Maru and Verdugo, Paula},
  journal={arXiv preprint arXiv:2301.07801},
  year={2023}
}

@article{Jar97,
  title={A closed model structure for differential graded algebras},
  author={Jardine, J.F.},
  journal={Fields Institute Communications},
  volume={17},
pages={55--58},
  year={1997}
}

@article{MR,
  title={Equivariant homotopy theory via simplicial coalgebras},
  author={Mart{\'\i}nez Alberga, Sof{\'\i}a and Rivera, Manuel},
  journal={Proceedings of the American Mathematical Society},
  volume={153},
  number={10},
  pages={4517--4534},
  year={2025}
}

@article {RR22,
    AUTHOR = {Raptis, George and Rivera, Manuel},
     TITLE = {The {S}implicial {C}oalgebra of {C}hains {U}nder {T}hree
              {D}ifferent {N}otions of {W}eak {E}quivalence},
   JOURNAL = {Int. Math. Res. Not. IMRN},
  FJOURNAL = {International Mathematics Research Notices. IMRN},
      YEAR = {2024},
    NUMBER = {16},
     PAGES = {11766--11811},
      ISSN = {1073-7928},
   MRCLASS = {Prelim},
  MRNUMBER = {4789098},
       DOI = {10.1093/imrn/rnae031},
       URL = {https://doi.org/10.1093/imrn/rnae031},
}

@article{goerss1995simplicial,
  title={Simplicial chains over a field and p-local homotopy theory},
  author={Goerss, Paul G},
  journal={Mathematische Zeitschrift},
  volume={220},
  number={1},
  pages={523--544},
  year={1995},
  publisher={Springer}
}

@article{scull,
  title={A model category structure for equivariant algebraic models},
  author={Scull, Laura},
  journal={Transactions of the American Mathematical Society},
  volume={360},
  number={5},
  pages={2505--2525},
  year={2008}
}

@article{bergner,
  title={Equivalence of models for equivariant $(\infty, 1)$-categories},
  author={Bergner, Julia E},
  journal={Glasgow Mathematical Journal},
  volume={59},
  number={1},
  pages={237--253},
  year={2017},
  publisher={Cambridge University Press}
}

@article{88,
  title={Resolutions of unbounded complexes},
  author={Spaltenstein, Nicolas},
  journal={Compositio Mathematica},
  volume={65},
  number={2},
  pages={121--154},
  year={1988}
}

@book{93,
  title={Hopf algebras and their actions on rings},
  author={Montgomery, Susan},
  number={82},
  year={1993},
  publisher={American Mathematical Soc.}
}

@article{21,
  title={Contramodules},
  author={Positselski, Leonid},
  journal={Confluentes Mathematici},
  volume={13},
  number={2},
  pages={93--182},
  year={2021}
}

@article{beilinson1988koszul,
  title={Koszul duality},
  author={Beilinson, Alexander A and Ginsburg, Victor A and Schechtman, Vadim V},
  journal={Journal of geometry and physics},
  volume={5},
  number={3},
  pages={317--350},
  year={1988},
  publisher={Elsevier}
}

@article{hinich,
  title={DG coalgebras as formal stacks},
  author={Hinich, Vladimir},
  journal={Journal of pure and applied algebra},
  volume={162},
  number={2-3},
  pages={209--250},
  year={2001},
  publisher={Elsevier}
}

@book{MM02,
  title={Equivariant orthogonal spectra and $ S $-modules},
  author={Mandell, Michael A and May, J Peter},
  year={2002},
  publisher={American Mathematical Soc.}
}
\end{document}